\documentclass[a4paper,oneside,11pt, reqno]{amsart}

\usepackage{mathrsfs}
\usepackage{amsfonts,amssymb,amsmath,amsthm}
\usepackage{url}
\usepackage{enumerate}
\usepackage{enumitem}
\usepackage{booktabs}
\usepackage{tabularx}
\usepackage{array}
\usepackage{diagbox}
\usepackage{float}
\usepackage{caption}

\usepackage{bbm}
\usepackage{float}
\usepackage{pstricks}
\usepackage{amscd}
\usepackage{amsmath}
\usepackage{amsxtra}
\usepackage[T1]{fontenc}
\usepackage[utf8]{inputenc}
\usepackage{lipsum}
\usepackage{tikz}
\usetikzlibrary{arrows}
\usetikzlibrary{calc}
\usepackage{hyperref}

\theoremstyle{plain}
\newtheorem{theorem}{Theorem}[section]

\newtheorem{lemma}[theorem]{Lemma}
\newtheorem{proposition}[theorem]{Proposition}
\newtheorem{corollary}[theorem]{Corollary}

\theoremstyle{definition}

\newtheorem{question}[theorem]{Question}

\theoremstyle{remark}
\newtheorem{remark}[theorem]{Remark}

\newtheorem{case-new}{Case}

\numberwithin{equation}{section}

\usepackage{etoolbox}

\makeatletter
\@namedef{subjclassname@2020}{%
  \textup{2020} Mathematics Subject Classification}
\makeatother

\usepackage{mathrsfs}
\usepackage{epsfig}
\usepackage[active]{srcltx}

\newcommand{\ncom}{\newcommand}
\ncom{\ul}{\underline}
\ncom{\ol}{\overline}
\ncom{\bq}{\begin{equation}}
\ncom{\eq}{\end{equation}}
\ncom{\beqn}{\begin{eqnarray*}}
\ncom{\eeqn}{\end{eqnarray*}}
\ncom{\beq}{\begin{eqnarray}}
\ncom{\eeq}{\end{eqnarray}}
\ncom{\nno}{\nonumber}
\ncom{\rar}{\rightarrow}
\ncom{\Rar}{\Rightarrow}
\ncom{\noin}{\noindent}
\ncom{\bc}{\begin{centre}}
\ncom{\ec}{\end{centre}}
\ncom{\sz}{\scriptsize}
\ncom{\rf}{\ref}
\ncom{\sgm}{\sigma}
\ncom{\Sgm}{\Sigma}
\ncom{\dt}{\delta}
\ncom{\Dt}{Delta}
\ncom{\lmd}{\lambda}
\ncom{\Lmd}{\Lambda}
\ncom{\eps}{\epsilon}
\ncom{\pcc}{\stackrel{P}{>}}
\ncom{\dist}{{\rm\,dist}}
\ncom{\sspan}{{\rm\,span}}
\ncom{\im}{{\rm Im\,}}
\ncom{\sgn}{{\rm sgn\,}}
\ncom{\ba}{\begin{array}}
\ncom{\ea}{\end{array}}
\ncom{\eop}{\hfill{{\rule{2.5mm}{2.5mm}}}}
\ncom{\eoe}{\hfill{{\rule{1.5mm}{1.5mm}}}}
\ncom{\eof}{\hfill{{\rule{1.5mm}{1.5mm}}}}
\ncom{\hone}{\mbox{\hspace{1em}}}
\ncom{\htwo}{\mbox{\hspace{2em}}}
\ncom{\hthree}{\mbox{\hspace{3em}}}
\ncom{\hfour}{\mbox{\hspace{4em}}}
\ncom{\hsev}{\mbox{\hspace{7em}}}
\ncom{\vone}{\vskip 2ex}
\ncom{\vtwo}{\vskip 4ex}
\ncom{\vonee}{\vskip 1.5ex}
\ncom{\vthree}{\vskip 6ex}
\ncom{\vfour}{\vspace*{8ex}}
\ncom{\norm}{\|\;\;\|}
\ncom{\integ}[4]{\int_{#1}^{#2}\,{#3}\,d{#4}}
\ncom{\inp}[2]{\langle{#1},\,{#2} \rangle}
\ncom{\Inp}[2]{\Langle{#1},\,{#2} \Langle}
\ncom{\vspan}[1]{{{\rm\,span}\#1 \}}}
\ncom{\dm}[1]{\displaystyle {#1}}

\usepackage{ulem}



\begin{document}
\title[Cyclic polynomials in
Dirichlet-type spaces]{Cyclic polynomials in 
Dirichlet-type Spaces of the unit bidisk}

\author[R. Nailwal]{Rajkamal Nailwal${}^1$}
\address{Rajkamal Nailwal, Institute of Mathematics, Physics and Mechanics, Ljubljana, Slovenia.}
\email{rajkamal.nailwal@imfm.si, raj1994nailwal@gmail.com}
\thanks{${}^1$Supported by the ARIS (Slovenian Research and Innovation Agency)
		research core funding No.\ P1-0288 and grant No.\ J1-60011.}

\author[A. Zalar]{Alja\v z Zalar${}^2$}
%\author[A. Zalar]{Alja\v z Zalar${}^{2,Q}$}
\address{Alja\v z Zalar, 
University of Ljubljana,
Faculty of Computer and Information Science  \& 
Faculty of Mathematics and Physics, \&
Institute of Mathematics, Physics and Mechanics, Ljubljana, Slovenia.}
\email{aljaz.zalar@fri.uni-lj.si}
\thanks{${}^2$Supported by the ARIS (Slovenian Research and Innovation Agency)
		research core funding No.\ P1-0288 and grants No.\ J1-50002, J1-60011, J1-70017.}
        
\begin{abstract} 
 For $\alpha \in \mathbb{R},$ we consider the scale of function spaces, namely the Dirichlet-type space  $\mathcal{D}_{\alpha}$ consisting of holomorphic functions  on the unit bidisk $\mathbb{D}^2$,   $f(z,w)=\sum_{k,l=0}^{\infty}a_{kl}z^kw^l$ such that 
$$\sum_{k,l=0}^{\infty}(k+l+1)^\alpha|a_{kl}|^2 < \infty.$$
In this paper, we solve an open problem
 posed by Torkinejad Ziarati concerning the cyclicity of the polynomial
\(2-z_1-z_2\) in \(\mathcal D_\alpha\) for \( \frac32 < \alpha \leq 2\).
We provide an affirmative answer and, as a consequence, complete the characterization
of cyclic polynomials in \(\mathcal D_\alpha\).
\end{abstract}
\maketitle

\section{Introduction}

Let $\mathbb{C}$ denote the complex plane, $\mathbb{D}=\{z \in \mathbb C, |z|< 1\}$ the open unit disk
and $\mathbb T=\{z \in \mathbb C, |z|= 1\}$ the unit circle in the complex plane. Given a polynomial $p \in \mathbb{C}[z_1,z_2]$, its bidegree is the pair $(m,n)$, 
where $m$ is the highest degree of $p$ in the variable $z_1$, 
and $n$ is the highest degree of $p$ in the variable $z_2.$ We write
$$\mathcal{Z}(p):=\{(z_1,z_2)\in \mathbb C^2\colon p(z_1,z_2)=0\}$$ 
for the vanishing set of $p.$
A nonzero polynomial $p$ is said to be \textit{irreducible} if 
 $p = qr$ with $q,r \in \mathbb{C}[z_1,z_2]$
implies that $q\in \mathbb C$ or $r\in \mathbb C$.

In the classical Hardy space $H^{2}(\mathbb{D})$, a function $f$ is called 
\textit{cyclic} if the smallest closed, shift-invariant subspace generated by its 
polynomial multiples coincide with the whole space. By Beurling's theorem, 
cyclic functions in this setting are precisely the outer functions, i.e., $f(0)\neq 0$ and
$$\log |f(0)|=\int_{0}^{2\pi}\log|f(e^{i\theta})|\frac{d\theta}{2\pi},$$
making the 
theory both elegant and complete.
In contrast, for the Hardy space on the 
bidisk $H^{2}(\mathbb{D}^{2})$, cyclic functions are outer but there exists an outer function which is not cyclic \cite{R1969}. With the present understanding, a characterization of cyclic functions seems to be a harder problem in several variables.  However, cyclic polynomials are characterized in $H^2(\mathbb D^n)$; they are precisely those polynomials which do not have zeros on the polydisk $ \mathbb D^n$ \cite{NGN1970}. This gap between the univariate and multivariate settings 
motivates investigations of cyclicity in other spaces, such as the 
Dirichlet space and the Dirichlet-type spaces.

In the Dirichlet space $D$ of the unit disk, Brown and Shields conjectured \cite[Question~12]{BS1984} that a function $f\in  D$ is cyclic if and only if it is outer and its boundary zero set has logarithmic capacity zero. They were able to establish the forward direction, while the converse, despite several attempts, remains open till now. Several partial results are known \cite{FKR2009,FEK2016}. The Brown–Shields conjecture continues to be a central open problem, motivating further exploration of cyclicity in Dirichlet-type spaces \cite{BCLSS2015I, BCLSS2015, BKKLSS2016, KKRS2019}.

\subsection{Dirichlet-type spaces} We now introduce the following Dirichlet-type space of the unit bidisk, where we will investigate the cyclicity of polynomials.
For $\alpha\in \mathbb R,$ the Dirichlet-type space denoted by $\mathcal D_{\alpha}$ on $\mathbb{D}^2$ consists of all holomorphic functions $f(z,w)=\sum_{k,l=0}^{\infty}a_{kl}z^kw^l$  such that 
 $$\|f\|_{\alpha}^2:=\sum_{k,l=0}^{\infty}(k+l+1)^\alpha|a_{kl}|^2 < \infty.$$
Note that for $\alpha=0,$ we recover the Hardy space $H^2({\mathbb D^2})$ of the unit bidisk and for $\alpha=1,$ the space $\mathcal{D}_{1}$ was introduced in \cite{BCG2025} in connection with toral $2$-isometries. This space has also been considered by Torkinejad Ziarati in \cite{Z25}, where the weight $(k+l+1)^\alpha$ is replaced by $(k+l+2)^\alpha$. However, it can be seen that these two norms are equivalent. Indeed, for $\alpha\geq 0$ we have
$$(k+l+1)^\alpha\leq (k+l+2)^\alpha\leq 2^\alpha(k+l+1)^\alpha,$$
while for $\alpha\leq 0$ it holds that
$$2^{\alpha}(k+l+1)^\alpha\leq (k+l+2)^\alpha\leq (k+l+1)^\alpha.$$

Investigations in this paper are motivated by the cyclicity results \cite{BCLSS2015, BKKLSS2016} 
obtained for the Dirichlet-type space $\mathfrak D_{{\alpha}},$ 
$\alpha\in \mathbb R$, on $\mathbb{D}^2$, 
which consists of holomorphic functions $f(z,w)=\sum_{k,l=0}^{\infty}a_{kl}z^kw^l$ such that 
\begin{equation} 
\label{def:isotropic-norm}
\|f\|_{\mathfrak D_{{\alpha}}}^2:=\sum_{k,l=0}^{\infty}(k+1)^{\alpha}(l+1)^{\alpha}|a_{kl}|^2< \infty.
\end{equation}
 
From the norm definitions, it is straightforward to see that for $\alpha \geq 0$ we have 
\begin{equation}
\label{inclusions-1}
\|\cdot\|_{\alpha}\leq \|\cdot\|_{\mathfrak D_{\alpha}} \leq \|\cdot\|_{2\alpha} \quad (\mathcal{D}_{2\alpha} \subseteq \mathfrak{D}_{\alpha} \subseteq \mathcal{D}_{\alpha}),
\end{equation}
and for $\alpha \leq 0$,
\begin{equation}
\label{inclusions-2}
\|\cdot\|_{2\alpha}\leq \|\cdot\|_{\mathfrak D_{\alpha}} \leq \|\cdot\|_{\alpha} \quad (\mathcal{D}_{\alpha} \subseteq \mathfrak{D}_{\alpha} \subseteq \mathcal{D}_{2\alpha}). 
\end{equation}
Inequalities \eqref{inclusions-1}, \eqref{inclusions-2} allow us to transfer properties between both kinds of Dirichlet-type spaces. In particular, we focus here on the notion of cyclicity.

Note that both spaces $\mathcal D_\alpha$ and $\mathfrak D_\alpha$ can be viewed as a generalization of the univariate Dirichlet space $D_{\alpha}, \alpha  \in \mathbb R$, 
 which consists of holomorphic functions $f: \mathbb D \rightarrow \mathbb C$, $f(z)=\sum_{k=0}^{\infty} a_kz^k,$
such that
\begin{equation}
\label{norm-univariate}
    \|f\|_{D_{\alpha}}^2:=\sum_{k=0}^{\infty} (k+1)^{\alpha}|a_k|^2<\infty.
\end{equation}

\subsection{Cyclic functions} In this subsection, we explain the motivation for studying cyclic functions in 
the Dirichlet-type space $\mathcal D_\alpha.$
Let $(M_{z_1},M_{z_2})$ be a pair of shift operators 
 acting on  $\mathcal{D}_{\alpha}$, i.e.,   
\[
(M_{z_1}f)(z_{1},z_{2}) := z_{1}f(z_{1},z_{2}), 
\quad 
(M_{z_2}f)(z_{1},z_{2}) := z_{2}f(z_{1},z_{2}),\quad f \in \mathcal{D}_{\alpha}.
\]
It is straightforward to verify that both $M_{z_1}$ and $M_{z_2}$ are bounded 
linear operators on $\mathcal{D}_{\alpha}$. From the operator-theoretic point 
of view, an important problem is to describe the closed subspaces of 
$\mathcal{D}_{\alpha}$ that are invariant under these shifts, namely those 
$\mathcal{M} \subseteq \mathcal{D}_{\alpha}$ for which  
\[
M_{z_1}\mathcal{M} \subseteq \mathcal{M}
\qquad  \text{and} \qquad
M_{z_2}\mathcal{M} \subseteq \mathcal{M}.
\]
A key step towards this description is to understand when a function 
$f \in \mathcal{D}_{\alpha}$ is cyclic, i.e., when the closed linear span  
\[
[f] := \overline{\operatorname{span}}\{ z_{1}^{k} z_{2}^{\ell} f : k,\ell \geq 0 \}
\]
coincides with the entire space $\mathcal{D}_{\alpha}$. It is clear from the definition that $[f]$ is the smallest closed subspace that contains $f$ and is invariant under the shift operators $M_{z_1}$ and $M_{z_2}.$ Clearly, at least one 
cyclic vector always exists, e.g., the constant function 
$f(z_{1},z_{2}) \equiv 1$ is cyclic, because polynomials in two variables are 
dense in $\mathcal{D}_{\alpha}$. In the next section, we will see that a necessary 
condition for cyclicity is that $f$ has no zeros in 
$\mathbb{D}^{2}$.  

Note that if $g \in [f]$, then $[g]\subseteq [f].$ Thus to check $f$ is cyclic in $\mathcal{D}_{\alpha},$ it suffices to show that there exists a 
sequence of polynomials $p_{n} \in \mathbb{C}[z_1,z_2]$ such that  
\[
\| p_{n} f - 1 \|_{\alpha} \to 0 \quad \text{as } n \to \infty.
\]

\subsection{Cyclic polynomials in $\mathfrak{D}_{\alpha}$} 
Recent work of Bénéteau et al.  \cite[Theorem]{BKKLSS2016}
provides a complete characterization of cyclic polynomials in $\mathfrak D_\alpha, \alpha \in \mathbb{R}$, on the bidisk (see \cite{KKRS2019} for the anisotropic setting
{$\mathfrak{D}_{(\alpha_1,\alpha_2)}$ obtained by replacing the weights $(k+1)^\alpha (l+1)^\alpha$ in \eqref{def:isotropic-norm} with the weights $(k+1)^{\alpha_1}(l+1)^{\alpha_2}$} and  \cite{KV2023} for the unit ball in $\mathbb C^2$). Their main result shows that the cyclicity of an irreducible
polynomial depends intricately on the structure of its zero set on the distinguished
boundary $\mathbb T^2$.
In particular, while non-vanishing in the bidisk is necessary for cyclicity,
additional restrictions on the boundary zero set become decisive when the
parameter $\alpha$ of the Dirichlet-type space $\mathfrak D_{\alpha}$ lies in the range $(\frac{1}{2}, \infty)$.  We recall the result for the reader's convenience.

\begin{theorem}[{\cite[Theorem]{BKKLSS2016}}] \label{Ben-cyc}
Let \( p \in \mathbb C[z_1,z_2] \) be an irreducible polynomial with no zeros in the bidisk.  We have the following:
\begin{enumerate}
    \item If \( \alpha \leq \frac{1}{2}, \) then \( p \) is cyclic in \( \mathfrak{D}_{\alpha}\).
     \item If \( \frac{1}{2}< \alpha \leq 1 \), then \( p \) is cyclic in \( \mathfrak{D}_{\alpha} \) if and only if \( \mathcal{Z}(p) \cap \mathbb{T}^2 \) is empty or finite or $p$ is a constant multiple of $\zeta-z_1$ or of $\zeta-z_2$ for some $\zeta \in \mathbb T$.
    
    \item If \( \alpha > 1 \), then \( p \) is cyclic in \( \mathfrak{D}_{\alpha} \) if and only if \( \mathcal{Z}(p) \cap \mathbb{T}^2 \) is empty.
\end{enumerate}
\end{theorem}

\subsection{Cyclic polynomials in $\mathcal{D}_{\alpha}$} The question of characterization of cyclic polynomials in $\mathcal{D}_{\alpha}$ was very recently studied by Torkinejad Ziarati \cite{Z25}, leaving one case open which depends on the following question posed by the author (see \cite[Open problem~1]{Z25}). 

\begin{question}\label{open}
    For $\frac{3}{2}\leq \alpha \leq 2,$ determine whether the polynomial $2-z_1-z_2$ is cyclic in $\mathcal D_{\alpha}.$
\end{question}
In this paper, we answer
{ Question \ref{open}, which, together with \cite[Theorem~31]{Z25}, completes the characterization of cyclic polynomials in $\mathcal D_{\alpha}$. We also provide alternative proofs for several results from \cite[Section~5]{Z25}.
}

\begin{theorem}[]\label{main}
Let \( p \in \mathbb C[z_1,z_2] \) be an irreducible polynomial with no zeros in the bidisk. We have the following:
\begin{enumerate}
    \item\label{main-pt1} If \( \alpha \leq 1, \) then \( p \) is cyclic in \( \mathcal{D}_{\alpha}\).
     \item\label{main-pt2} If \( 1< \alpha \leq 2 \), then \( p \) is cyclic in \( \mathcal{D}_{\alpha} \) if and only if \( \mathcal{Z}(p) \cap \mathbb{T}^2 \) is empty or finite.
    
    \item\label{part-3} If \( \alpha > 2 \), then \( p \) is cyclic in \( \mathcal{D}_{\alpha} \) if and only if \( \mathcal{Z}(p) \cap \mathbb{T}^2 \) is empty.
\end{enumerate}
\end{theorem} 

\begin{remark}
Theorem \ref{main}$(iii)$ follows directly from Theorem \ref{Ben-cyc}$(iii)$ 
{together with \eqref{inclusions-1}.}  
In contrast, establishing parts $(i)$ and $(ii)$ of Theorem \ref{main} is less straightforward. For the proof of part $(i)$, we refer to \cite{Z25}.
The proof of part $(ii)$ relies on the techniques used in the proof of \cite[Theorem 3.1]{BKKLSS2016}, together with arguments from \cite[Appendix A]{KKRS2019}.
\end{remark}

\subsection{
Cyclic polynomials in 
$\mathfrak D_\alpha$
versus
$\mathcal D_\alpha$
}

The following table summarizes the conditions for an irreducible polynomial \(p\) to be cyclic in the spaces \(\mathfrak D_{\alpha}\) and \(\mathcal D_{\alpha}\).

\begin{center}
\small
\renewcommand{\arraystretch}{1.25}
\captionof{table}{Assume that $p\in \mathbb C[z_1,z_2]$ is an irreducible polynomial with no zeroes on $\mathbb D^2$. The table gives conditions on the set $\mathcal Z(p)\cap \mathbb T^2$ for $p$ to be cyclic in the product-weight space $\mathfrak D_{\alpha}$ and the sum-weight space $\mathcal D_\alpha$ for different choices of $\alpha$.}
\medskip

\begin{tabularx}{\textwidth}{|>{\centering\arraybackslash}m{2.2cm}|>{\raggedright\arraybackslash}X|>{\raggedright\arraybackslash}X|}
\hline
\diagbox[width=2.2cm]{$\alpha$}{space}
& ${\mathfrak D_\alpha}$ 
\begin{footnotesize}
(weights $(k+1)^\alpha(l+1)^\alpha$)
\end{footnotesize}
& ${\mathcal D_\alpha}$ 
\begin{footnotesize}
(weights $(k+l+1)^\alpha$)
\end{footnotesize} \\
\hline
$\alpha \leq \frac{1}{2}$ 
& no condition 
& no condition \\
\hline
$\frac{1}{2} < \alpha \leq 1$ 
& $\mathcal Z(p)\cap\mathbb T^2$ is empty or finite, or $p$ is a constant multiple of $\zeta-z_1$ or $\zeta-z_2$ for some $\zeta\in\mathbb T$
& no condition \\
\hline
$1 < \alpha \leq 2$ 
& $\mathcal Z(p)\cap\mathbb T^2=\varnothing$
& $\mathcal Z(p)\cap\mathbb T^2$ is empty or finite \\
\hline
$\alpha > 2$ 
& $\mathcal Z(p)\cap\mathbb T^2=\varnothing$
& $\mathcal Z(p)\cap\mathbb T^2=\varnothing$ \\
\hline
\end{tabularx}
\end{center}
\medskip

\begin{remark}
The table highlights a clear difference between the cyclicity criteria in the two families of spaces. 
For the product-weight spaces $\mathfrak D_\alpha$, the transition occurs already at $\alpha=\frac12$: when $\alpha\leq \frac12$, every irreducible polynomial with no zeros in $\mathbb D^2$ is cyclic, whereas for $\frac12<\alpha\leq 1$ one must already impose restrictions on the set $\mathcal Z(p)\cap \mathbb T^2$. In that range, however, the spaces $\mathfrak D_\alpha$ still do not distinguish between a polynomial with finitely many zeros on $\mathbb T^2$ and the special factors $1-z_1$ and $1-z_2$, which have infinitely many zeros on $\mathbb T^2$. In fact, the cyclicity of these latter examples was used in \cite{BKKLSS2016} to deduce cyclicity for polynomials with finitely many boundary zeros.

By contrast, for the sum-weight spaces $\mathcal D_\alpha$, no boundary condition is needed for the entire range $\alpha \leq 1$. Thus, compared with $\mathfrak D_\alpha$, cyclicity persists in $\mathcal D_\alpha$ up to the larger threshold $\alpha = 1$. A second difference appears for $1 < \alpha \leq 2$: in $\mathcal D_\alpha$, cyclicity still holds whenever $\mathcal Z(p) \cap \mathbb T^2$ is finite, while in $\mathfrak D_\alpha$ one already requires $\mathcal Z(p) \cap \mathbb T^2 = \varnothing$. Finally, for $\alpha > 2$ the two scales of spaces exhibit the same behavior, namely that cyclicity is equivalent to the absence of zeros on $\mathbb T^2$.

In particular, the table shows that the spaces $\mathcal D_\alpha$ are, in this sense, more permissive than the spaces $\mathfrak D_\alpha$: finite intersections with $\mathbb T^2$ remain admissible for a larger range of parameters, and the exceptional role of the factors $1-z_1$ and $1-z_2$ disappears.
\end{remark}

\subsection{Organization of the paper} In Section \ref{preli}, we recall some definitions and necessary results needed to prove our main result. In Section \ref{cyc-2-z-1}, we present a solution to Question~\ref{open}. In Section \ref{p-o-m}, we provide a proof of Theorem~\ref{main}. This is done by first proving the case of cyclicity of the polynomial having at most finitely many zeros on $\mathbb{T}^2$ in $\mathcal D_{\alpha},\alpha \leq 2$. We then complete the proof of Theorem~\ref{main} using the results of the preceding sections. In Section~\ref{gpcf}, we present a few necessary conditions for the cyclicity of a function in $\mathcal D_{\alpha}, \alpha \in \mathbb R.$
In Section~\ref{con-rem}, we conclude the paper with a discussion on the cyclicity of $z_i - a,$ with $|a|\geq1,$ $1 - z_1 z_2$ and a result based on capacity.

\section{Preliminaries}\label{preli}
In this section, we list some properties of cyclic functions which are needed to give a self-contained treatment of the proof of Theorem~\ref{main} (see Subsection \ref{prel:properties}) and recall an inequality on the value of a real analytic function (see Subsection \ref{subsec:loj}). 

\subsection{Some properties of cyclic functions in $\mathcal D_\alpha$}
\label{prel:properties}
Note that $\mathcal{D}_{\alpha}, \alpha \in \mathbb R,$ is a reproducing kernel Hilbert space, i.e., the evaluation map $e_w$, for $w \in \mathbb D^2,$ $e_{w}(f):=f(w)$ is a continuous linear functional on $\mathcal{D_{\alpha}}$ and its multiplier space is defined as
$$M(\mathcal{D_{\alpha}})=\{\phi:\mathbb D^2\rightarrow \mathbb C: \phi f\in \mathcal{D}_{\alpha} \,\; \text{for all}  \,  f\in \mathcal{D_{\alpha}}\}.$$
Elements of $M(\mathcal{D_{\alpha}})$ are called \textit{multipliers} of $\mathcal{D}_{\alpha}.$ It is easy to verify that all polynomials are multipliers of $\mathcal{D}_{\alpha}, \alpha \in \mathbb R.$

Some properties of cyclic functions are the following:
\begin{enumerate}[leftmargin=0.6cm]
\item A cyclic function in $\mathcal D_\alpha$ cannot vanish on the bidisk.  To see this, take $p_n$ to be a sequence of polynomials such that $\|p_nf-1\|_{\alpha}\rightarrow 0.$ Since evaluations are continuous, the conclusion follows from the following expression
\beqn
p_n(z_1,z_2)f(z_1,z_2)-1=e_{(z_1,z_2)}(p_nf-1).
\eeqn

\item Given a cyclic function $f \in \mathcal{D}_{\alpha}, \alpha \in \mathbb R,$ the function defined by $g(z_1,z_2):=f(\zeta z_1,\eta z_2),$ where $ \zeta, \eta \in \mathbb T$, is clearly also cyclic in  $\mathcal{D}_{\alpha}$. Indeed, if there exists a sequence of polynomials $\{p_{n}\}_{n \in \mathbb Z_+}$ such that
\begin{equation}    
    \label{convergence}
        \|p_nf-1\|_{\alpha}\rightarrow 0,
\end{equation}
then the sequence defined by $q_{n}(z_1,z_2):=p_{n}(\zeta z_1,\eta z_2)$ satisfies $$\|q_ng-1\|_{\alpha}=\|p_nf-1\|_{\alpha}\rightarrow 0,$$
proving the cyclicity of $g$.
\item 
Assume that $f$ is a reducible polynomial with $f=gh$ for some nonconstant polynomials $g$ and $h$. 
Then $f=gh$ is cyclic in $\mathcal D_\alpha$ 
if and only if $g$ and $h$ are cyclic in $\mathcal D_\alpha$.
Let us verify:

If $f$ is cyclic, then there is a sequence of polynomials $\{p_{n}\}_{n \in \mathbb Z_+}$ satisfying \eqref{convergence}.
But then the sequence $\{r_{n}\}_{n \in \mathbb Z_+}$ and 
$\{s_{n}\}_{n \in \mathbb Z_+}$, where $r_n:=p_ng$ and $s_n:=p_nh$,
satisfy $\|r_nh-1\|_{\alpha}\rightarrow 0$ and 
$\|s_ng-1\|_{\alpha}\rightarrow 0$, proving cyclicity of $g$ and $h$.

Conversely, assume that $g$ and $h$ are cyclic. Then there exists a sequence of polynomials $\{r_{n}\}_{n \in \mathbb Z_+}$ such that 
$\|r_ng-1\|_{\alpha}\rightarrow 0.$ Note that 
\begin{align*}
\|r_ngh-h\|_{\alpha} \leq \|M_h\|\|r_ng-1\|_{\alpha}
\end{align*}
where $\|M_{h}\|$ is the operator norm of the multiplication operator $M_{h}.$
This shows that $h \in [f].$ Hence $[h]\subseteq [f].$ Since $h$ is cyclic, it follows that $f$ is cyclic. 

Therefore, it suffices to characterize cyclicity of irreducible polynomials in $\mathcal D_\alpha$.
\end{enumerate}

\subsection{{\L}ojasiewicz’s inequality}
\label{subsec:loj}

In the proof of Theorem \ref{thm:no-zeros} below, the following inequality will be used essentially.

\begin{theorem}[{\cite[{\L}ojasiewicz’s inequality]{KP2002}}]
\label{lojas-ineq}
    Let $f$ be a nonzero real analytic function on an open set $U \subseteq \mathbb R^n.$ Assume the zero set $\mathcal Z(f)$ of $f$ in $U$ is nonempty. Let $E$ be a compact subset of $U.$ Then there are constants $C>0$ and $q \in \mathbb N,$ depending on $E,$ such that
    $$|f(x)|\geq C\cdot \dist(x, \mathcal Z(f))^q$$
    for every $x \in E.$
\end{theorem}

\section{Cyclicity of the polynomial $2-z_1-z_2$}\label{cyc-2-z-1}
In this section, we answer Question~\ref{open} using the techniques from \cite[Appendix~A]{KKRS2019}.
\begin{theorem}\label{cyc-1-1}
    Let $\alpha\leq 2$. Then $p(z_1,z_2)=2-z_1-z_2$ 
    is cyclic in $\mathcal{D}_{\alpha}$. 
\end{theorem}
\begin{proof}
 It suffices to show  that $p$ is cyclic in $\mathcal{D}_{2}.$ Let $f\in \mathcal D_2$ 
 be such that $f\perp [p]$.
 Consider the following series of $f,$
 $$f(z_1,z_2)=\sum_{i,j=0}^\infty\frac{b_{i,j}}{(i+j+1)^2}z_1^iz_2^j.$$
 We will show that $f=0$, or equivalently $b_{i,j}=0$ 
 for $i, j \in \mathbb Z_+.$ 
 By $f\perp [p]$, it follows that
 \beq\label{rec-rel}
 2b_{k,l}=b_{k+1,l}+b_{k,l+1}, \quad k,l \in \mathbb Z_+.
 \eeq
 Since $f\in \mathcal D_2$, a new function $$g(z_1,z_2):=\sum_{i,j\geq 0}b_{i,j}z_1^iz_2^j$$ 
 belongs to $\mathcal D_{-2}$, and by \eqref{rec-rel},
 \beq \label{rel-g}(z_1+z_2-2z_1z_2)g(z_1,z_2)=z_1g(z_1,0)+z_2g(0,z_2), \quad  (z_1,z_2)\in \mathbb D^2.
 \eeq
 We next introduce the substitutions  
\begin{equation}
\label{zeta}
z_1 = \frac{\zeta}{\zeta - 1}, \qquad  
z_2 = \frac{\zeta}{\zeta + 1}.
\end{equation}
Observe that $z_1 \in \mathbb{D}$ precisely when $\Re \zeta < \tfrac{1}{2}$, and  
$z_2 \in \mathbb{D}$ precisely when $\Re \zeta > -\tfrac{1}{2}$,  
where $\Re \zeta$ denotes the real part of the complex number $\zeta$.
By substituting the above expressions for $z_1$ and $z_2$ into~\eqref{rel-g}, we get
\[
0 = \frac{\zeta}{\zeta - 1} g\!\left(\frac{\zeta}{\zeta - 1}, 0\right)
   + \frac{\zeta}{\zeta + 1} g\!\left(0, \frac{\zeta}{\zeta + 1}\right),
 \qquad \text{for } -\tfrac{1}{2} < \Re \zeta < \tfrac{1}{2}.
\]
Thus, defining $h : \mathbb{C} \to \mathbb{C}$ by  
\[
h(\zeta) =
\begin{cases}
\displaystyle \frac{1}{\zeta - 1}\, g\!\left(\frac{\zeta}{\zeta - 1}, 0\right),
& \text{if } \Re \zeta < \tfrac{1}{2},\\[1em]
\displaystyle -\frac{1}{\zeta + 1}\, g\!\left(0, \frac{\zeta}{\zeta + 1}\right),
& \text{if } \Re \zeta > -\tfrac{1}{2},
\end{cases}
\]
we obtain that $h$ is a well-defined entire function.
 Note that
\[
\sum_{k\ge0} \frac{|b_{k0}|^2}{(k+1)^2}
\;\asymp\;
\int_{\mathbb D} |g(z_1,0)|^2 (1 - |z_1|^2)\, dA(z_1).
\]
This can be verified by integrating the right-hand side using polar coordinates.
Thus, we have
\begin{align}
\label{approx-eq-1}
\begin{split}
\sum_{k\ge0} \frac{|b_{k0}|^2}{(k+1)^2}
&\asymp \int_{\Re \zeta < 1/2}
|(\zeta - 1) h(\zeta)|^2
\left( 1 - \left| \frac{\zeta}{\zeta - 1} \right|^2 \right)
\frac{dA(\zeta)}{|\zeta - 1|^4}\\
&= \int_{\Re \zeta < 1/2}
|h(\zeta)|^2
\frac{1 - 2\Re \zeta}{|\zeta - 1|^{4}}
\, dA(\zeta),
\end{split}
\end{align}
and similarly,
\begin{align*}
\sum_{l\ge0} \frac{|b_{0l}|^2}{(l+1)^{2}}
&\asymp
\int_{\mathbb D} |g(0,z_2)|^2 (1 - |z_2|^2)\, dA(z_2)\\
&= \int_{\Re \zeta > -1/2}
|(\zeta + 1) h(\zeta)|^2
\left( 1 - \left| \frac{\zeta}{\zeta + 1} \right|^2 \right)
\frac{dA(\zeta)}{|\zeta + 1|^4}
\\
&= \int_{\Re \zeta > -1/2}
|h(\zeta)|^2
\frac{1 + 2\Re \zeta}{|\zeta + 1|^{4}}
\, dA(\zeta).
\end{align*}
Both series are finite due to $g \in \mathcal{D}_{-2}$; 
hence, the sum of the two integrals is finite, and consequently,
\[
\int_{|\zeta| > 1} \frac{|h(\zeta)|^2}{|\zeta|^{4}}\, dA(\zeta) < \infty.
\]
This forces \( h \) to be a polynomial of degree at most $1$, 
i.e.,  
\begin{equation} 
\label{h-is-linear}
h(\zeta)=a(\zeta-1)+b,\quad a,b\in \mathbb{C}.
\end{equation}
Assume that $a\neq 0$. Using \eqref{h-is-linear} in \eqref{approx-eq-1}, we get
$$
\int_{\Re \zeta < 1/2}
\frac{1 - 2\Re \zeta}{|\zeta - 1|^{2}}
\, dA(\zeta)<\infty,
$$
which is a contradiction. Hence, $a=0$.
Note that
\begin{align}
\label{form-of-g}
\begin{split}
g(z_1,0) 
&\underbrace{=}_{\eqref{zeta}} \frac{1}{z_1 - 1}\, h\!\left(\frac{z_1}{z_1 - 1}\right)=\frac{b}{z_1-1},\\
g(0,z_2) 
&\underbrace{=}_{\eqref{zeta}} \frac{1}{z_2 - 1}\, h\!\left(\frac{z_2}{1-z_2 }\right)=\frac{b}{z_2-1}.
\end{split}
\end{align}
Using \eqref{form-of-g} in \eqref{rel-g}, we get
\begin{equation}
    (z_1+z_2-2z_1z_2)g(z_1,z_2)
    =\frac{bz_1}{z_1-1}+\frac{bz_2}{z_2-1}, \quad  (z_1,z_2)\in \mathbb D^2.
\end{equation}
{
 We have
\[
g(z_1,z_2)=-\frac{b}{(1-z_1)(1-z_2)}.
\]
For \(|z_1|<1\) and \(|z_2|<1\), the geometric series gives
\[
\frac{1}{1-z_1}=\sum_{k=0}^\infty z_1^k,
\qquad
\frac{1}{1-z_2}=\sum_{l=0}^\infty z_2^l,
\]
hence, by multiplication of absolutely convergent series,
\[
\frac{1}{(1-z_1)(1-z_2)}
=
\sum_{k=0}^\infty\sum_{l=0}^\infty z_1^k z_2^l.
\]
Therefore
\[
g(z_1,z_2)
=
-b\sum_{k,l=0}^\infty z_1^k z_2^l,
\]
so the Taylor coefficients of \(g\) satisfy
\[
a_{kl}=-b,\qquad k,l\ge 0.
\]
It follows that
\[
\|g\|_{-2}^2
=
\sum_{k,l=0}^{\infty}(k+l+1)^{-2}\,|a_{kl}|^2
=
|b|^2\sum_{k,l=0}^{\infty}\frac{1}{(k+l+1)^2}.
\]
We now group terms by \(n=k+l\):
\[
\sum_{k,l=0}^{\infty}\frac{1}{(k+l+1)^2}
=
\sum_{n=0}^{\infty}\sum_{\substack{k,l\ge 0\\ k+l=n}}\frac{1}{(n+1)^2}.
\]
For each \(n\ge 0\) there are exactly \(n+1\) pairs \((k,l)\in\mathbb{N}_0^2\) with \(k+l=n\). Hence
\[
\sum_{n=0}^{\infty}\sum_{\substack{k,l\ge 0\\ k+l=n}}\frac{1}{(n+1)^2}
=
\sum_{n=0}^{\infty}\frac{n+1}{(n+1)^2}
=
\sum_{n=0}^{\infty}\frac{1}{n+1}.
\]
The latter series diverges, and therefore
\[
\|g\|_{-2}=\infty.
\]
Thus, \(g \notin \mathcal D_{-2}\) unless \(b = 0\). In particular, if \(g \in \mathcal D_{-2}\), we must have \(b = 0\), which forces \(g = 0\). Therefore, \(h = 0\), and consequently \(f = 0\). This completes the proof.
 }
\end{proof}

\begin{remark}
    The proof shows that if a sequence of complex numbers $\{b_{k,l}\}_{k,l\in \mathbb Z_+}$ satisfies
    \eqref{rec-rel}
and
    $$\sum_{k,l \in \mathbb Z_+}\frac{|b_{k,l}|^2}{(k+l+1)^2}< \infty,$$
    then $b_{k,l}=0$ for all $k,l \in \mathbb Z_+.$
\end{remark}
The following result is an immediate and noteworthy consequence, 
{which will be used essentially in the proof of Theorem \ref{thm:no-zeros} below}.

\begin{corollary} \label{p-zeta-1-2}
    For $\zeta_1, \zeta_2 \in \mathbb T,$  $p(z_1,z_2)=2-\zeta_1z_1-\zeta_2z_2$ is cyclic in $\mathcal{D}_{\alpha} $ for $\alpha \leq 2.$ 
\end{corollary}
\begin{proof}
    This follows from Theorem~\ref{cyc-1-1} together with the fact that cyclicity is preserved under rotation.
\end{proof}

\section{A proof of Theorem~\ref{main}}
\label{p-o-m}
 First, we investigate the cyclicity of polynomials in $\mathcal{D}_{\alpha}$, $\alpha \le 2$, that possess only finitely many zeros on $\mathbb T^2$. In the previously studied settings (see \cite[Section~3]{BKKLSS2016}), the cyclicity of the polynomial $1 - z_i$ was essentially used in the proof of the main result \cite[Theorem 3.1]{BKKLSS2016}. However, in our framework, $1 - z_i$ needs to be replaced by another polynomial cyclic in $\mathcal{D}_{\alpha}$, $\alpha \leq 2$, since $1-z_i$ is cyclic in $\mathcal{D}_{\alpha}$ only for $\alpha \le 1$
 {(see \cite[Proposition~38]{Z25} or Proposition \ref{one-var} below).} It turns out that $2 - z_1 - z_2$ serves as a suitable substitute. 

\begin{theorem}
\label{thm:no-zeros}
 Consider a polynomial $p \in \mathbb C[z_1,z_2] $ having no zeros in $\mathbb D^2$ and finitely many on $\mathbb T^2.$ Then $p$ is cyclic in $\mathcal D_{\alpha}$ for $\alpha \leq 2.$
\end{theorem}

The proof of Theorem \ref{thm:no-zeros} will parallel the arguments in the proof of \cite[Theorem 3.2]{BKKLSS2016},
which deals with the cyclicity in $\mathfrak D_\alpha$,
using the modifications described in the paragraph before Theorem \ref{thm:no-zeros}.

The following result, which is an analog of 
\cite[Lemma 3.3]{BKKLSS2016}, will allow us to compare polynomials having finitely many zeros on $\mathbb T^2$ with the polynomials of the type $2-\zeta z_1-\eta z_2, \zeta, \eta \in \mathbb T,$ whose cyclicity has already been established in Corollary~\ref{p-zeta-1-2}.
\begin{lemma}\label{f-rep}
    Suppose $f \in \mathbb C[z_1,z_2]$ has no zeros in $\mathbb D^2$ and finitely many on $\mathbb T^2$, i.e., $\mathcal Z(f)\cap \mathbb T^2=\{(\zeta_j,\eta_j)\in \mathbb{T}^2,j =1,\ldots,k\}.$ 
    Then, for any integer $m$, there exists a sufficiently large $N$ such that the function
    $$Q(z_1,z_2)=\frac{
    \prod_{i=1}^k (2-\zeta_i^{-1}z_1- \eta_i^{-1}z_2)^N}{f(z_1,z_2)}$$
    is $m$-times differentiable on $\mathbb T^2.$
\end{lemma}
\begin{proof}
Let $\{(\zeta_j,\eta_j)\in \mathbb{T}^2$, 
$j=1,\ldots,k$ be as in the statement of the lemma and define the polynomials 
$$p_{j}(z_1,z_2)=2-{\zeta_j}^{-1}z_1-{\eta_j}^{-1}z_2.$$
Clearly, $p_j$ has only one zero $(\zeta_j, \eta_j)$ on $\mathbb T^2.$  Note that
\begin{equation} 
\label{inequality-on-p}
|p_{j}(z_1,z_2)|^2\leq |z_1-\zeta_j|^2+|z_2-\eta_{j}|^2+2|(z_1-\zeta_j)(z_2-\eta_{j})|.
\end{equation}
Writing $z_1,z_2, \zeta_j,\eta_j\in \mathbb T$ as 
$$z_1=e^{ix_1},\; 
z_2=e^{ix_2},\;
\zeta_j=e^{iy_{1,j}},\;
\eta_j=e^{iy_{2,j}},$$
where $x_1, x_2,y_{1,j},y_{2,j} \in [0,2\pi)$,
respectively, \eqref{inequality-on-p} becomes
\begin{equation}\label{p-inq}
\begin{aligned}
|p_{j}(e^{ix_1},e^{ix_2})|^2
&\le |e^{ix_1}-e^{iy_{1,j}}|^2 + |e^{ix_2}-e^{iy_{2,j}}|^2 \\
&\quad + 2\,\big|(e^{ix_1}-e^{iy_{1,j}})(e^{ix_2}-e^{iy_{2,j}})\big|.
\end{aligned}
\end{equation}
Define  
\[
r(x_1, x_2) = |f(e^{ix_1}, e^{ix_2})|^2, \qquad (x_1, x_2) \in \mathbb{R}^2,
\]
and let \( \mathcal{Z}(r) \) denote its zero set.  
Consider the compact set \( E = [0, 2\pi]^2 \).  
By Theorem \ref{lojas-ineq}, there exist a constant \( C > 0 \) and an integer \( q \in \mathbb{N} \) such that  
\[
r(x) \ge C\, \mathrm{dist}(x, \mathcal{Z}(r))^{q}, \qquad x \in E.
\]

Since the set \( \mathcal{Z}(r) \cap E \) consists of finitely many points, there is a constant \( c > 0 \) satisfying  
\begin{equation}\label{dist-x}
\mathrm{dist}(x, \mathcal{Z}(r))^{2} \ge c\, 
\prod_{y \in \mathcal{Z}(r) \cap E} |x - y|^{2}, 
\qquad x \in E.
\end{equation}
Also, we have
\beq\label{x-y}
|x-y|^2=\sum_{j=1,2}|x_j-y_j|^2\geq \sum_{j=1,2}|e^{ix_j}-e^{iy_j}|^2\geq 2\prod_{j=1,2}|e^{ix_j}-e^{iy_j}|.
\eeq
 Using \eqref{p-inq},\eqref{dist-x} and \eqref{x-y}, we have a constant $C_1>0$ such that
 \beqn
 \mathrm{dist}(x, \mathcal{Z}(r))^2 \ge  C_1 \prod_{j} |p_{j}|^2.
 \eeqn
This yields that the function  
\[
\frac{\prod_{j} |p_{j}|^{\,q}}{|f|^2},
\]
is bounded on \( \mathbb{T}^2 \).  We now apply the standard trick, which is 
to increase the exponent in the numerator and assign the value zero at the zeros of \( f \),  
to obtain a function that is $m$-times continuously differentiable on \( \mathbb{T}^2 \).
This completes the proof of Lemma \ref{f-rep}.
\end{proof}

Finally, we can prove Theorem \ref{thm:no-zeros}

\begin{proof}[Proof of Theorem \ref{thm:no-zeros}]
By Lemma~\ref{f-rep}, we obtain a function 
    $$Q(z_1,z_2)=\frac{
    \prod_{i=1}^k (2-\zeta_i^{-1}z_1- \eta_i^{-1}z_2)^N}{p(z_1,z_2)}=:\frac{g(z_1,z_2)}{p(z_1,z_2)},$$
where $(\zeta_i,\eta_i)\in \mathbb T^2$, which is
twice continuously differentiable on $\mathbb T^2$.
Thus its Fourier coefficients $\widehat{Q}(k,l)$ satisfy
$$\sum_{k,l}|\widehat{Q}(k,l)|^2(k+1)^2(l+1)^2< \infty.$$
But since $$\sum_{k,l}|\widehat{Q}(k,l)|^2(k+l+1)^2\leq \sum_{k,l}|\widehat{Q}(k,l)|^2(k+1)^2(l+1)^2,$$ we obtain $Q \in \mathcal{D}_{\alpha}, \alpha \leq 2.$ Hence $g(z_1,z_2)\in p \mathcal{D}_{\alpha}, \alpha \leq 2.$ Since $g$ is cyclic in $\mathcal{D}_2$ 
{by Corollary \ref{p-zeta-1-2}}
and $p$ is a multiplier, we obtain that $p$ is also cyclic in $\mathcal D_2.$
\end{proof}

\subsection{Proof of Theorem~\ref{main}}
\label{subsec:proof-of-main}

For completeness, we provide a proof of Theorem~\ref{main} in this section.

We recall the following result from \cite[p.\ 17]{Z25}. This establishes the cyclicity of polynomials having no zeros in the bidisk. The proof closely parallels the argument of \cite[Theorem 4.1]{BKKLSS2016} which deals with the cyclicity in $\mathfrak D_\alpha$, with a few modifications needed in the framework of $\mathcal D_\alpha$.

\begin{theorem} \label{gen-cyc}
    Assume that $\alpha\leq 1$. 
    Any polynomial $f\in \mathbb C[z_1,z_2]$ that does not vanish in the bidisk is cyclic in $\mathcal D_\alpha$.
\end{theorem}

The proof of Theorem~\ref{main} is now straightforward using the results above.

\begin{proof}[Proof of Theorem~\ref{main}]
{Observe that for $\alpha>0$,
\beq \label{ijal}(1+i)^{\alpha/2}(1+j)^{\alpha/2}\leq (1+i+j)^{\alpha}\leq (1+i)^{\alpha}(1+j)^{\alpha},\quad i,j \geq 0.\eeq}
We divide the argument into three cases according to the value of $\alpha$:
$\alpha>2$, $\alpha\in (1,2]$ and $\alpha\leq 1$.\\

 \noindent \textbf{Case 1:} $\alpha > 2$. 
First, assume that $p$ is cyclic in $\mathcal{D}_{\alpha}$. By the first inequality in \eqref{ijal}, it follows that $p$ is cyclic in $\mathfrak{D}_{\alpha/2}$. Since $\alpha/2 > 1$, Theorem~\ref{Ben-cyc}$(iii)$ implies that $\mathcal{Z}(p) \cap \mathbb{T}^2 = \varnothing$.

Now assume that $\mathcal{Z}(p)\cap \mathbb{T}^2 = \varnothing$. Then, by Theorem~\ref{Ben-cyc}$(iii)$, $p$ is cyclic in $\mathfrak{D}_{\alpha}$. Using the second inequality in \eqref{ijal}, we conclude that $p$ is cyclic in $\mathcal{D}_{\alpha}$.\\

% Now assume that $\mathcal Z(p)\cap \mathbb T^2$ is empty. Then by Theorem \ref{Ben-cyc}$(iii),$  $p$ is cyclic in $\mathfrak D_{\alpha}.$ Using the second inequality in \eqref{ijal}, we obtain   $p$ is cyclic in $\mathcal D_{\alpha}.$ \\

\noindent \textbf{Case 2:} $\alpha\in (1,2]$.
{The first inequality in \eqref{ijal} implies that if a polynomial $p$ is cyclic in $\mathcal{D}_{\alpha}$, then it is also cyclic in $\mathfrak{D}_{\alpha/2}$. By Theorem~\ref{Ben-cyc}$(ii)$, the zero set $\mathcal{Z}(p)\cap \mathbb{T}^2$ is either empty or finite, or $p$ is a constant multiple of $\zeta - z_1$ or $\zeta - z_2$ for some $\zeta \in \mathbb{T}$. However, by \cite[Proposition~38]{Z25}, the functions $\zeta - z_i$, with $\zeta \in \mathbb{T}$, are not cyclic in $\mathcal{D}_{\alpha}$ (see also Proposition~\ref{one-var}). This completes the proof of the forward implication in this case.}

For the converse, we have the following:\\

\noindent \textbf{Subcase 2.1:} $\mathcal Z(p)\cap \mathbb T^2=\varnothing$. Using the case $\alpha >2,$ we obtain that $p$ is cyclic in $\mathcal D_3$. Since $\|\cdot\|_{\alpha} \leq \|\cdot\|_{3}$ for $\alpha \leq 2,$ cyclicity in $\mathcal D_\alpha$ follows.\\

\noindent \textbf{Subcase 2.2:} $\mathcal Z(p)\cap \mathbb T^2$ is finite. This subcase follows from Theorem \ref{thm:no-zeros}.\\

\noindent \textbf{Case 3:} $\alpha \leq 1$.
This case follows from Theorem~\ref{gen-cyc}.
 \end{proof}   
 
\begin{remark}
\begin{enumerate}[leftmargin=0.6cm]
    \item The case $\alpha > 2$, 
    has been solved in 
    \cite[Theorem 31(c)]{Z25}
by proving that $\mathcal D_\alpha$ is an algebra. We have provided an alternative
proof.
    \item Note that if $p$ has no zeros in the bidisk and only finitely many zeros on $\mathbb{T}^2$, then its cyclicity in $\mathcal{D}_{\alpha}$ for $\alpha \le 1$ also follows from Theorem~\ref{Ben-cyc} by a direct comparison of norms, {similarly to Case~1 in the proof of Theorem~\ref{main} above}. However, this result excludes the cyclicity of a polynomial having infinitely many zeros on $\mathbb T^2$ as in the case of $1-z_1z_2.$ 
% \item {Earlier work on cyclicity included choices of $\alpha$ for which both polynomials with finitely many zeros on $\mathbb{T}^2$ and polynomials of the form $\zeta - z_i$, $\zeta \in \mathbb{T}$, whose zero sets are one-dimensional, were cyclic. In the present setting, however, these two types exhibit markedly different behavior.}
\end{enumerate}
\end{remark}

\subsection{An alternative approach to non-cyclicity in $\mathcal{D}_\alpha$ through its relation to the anisotropic Dirichlet-type spaces}
\label{subsec:alternative-approach}

We conclude this section by presenting an alternative approach, suggested by the referee, for proving the non-cyclicity of a polynomial in $\mathcal{D}_{\alpha}$ using results on cyclicity in the anisotropic Dirichlet-type spaces. This approach yields the following necessary conditions for cyclicity:
\begin{enumerate}[leftmargin=*]
 \item\label{point-1-through-anisotropic}  If $\alpha > 1$ and $p \in \mathbb C[z_1, z_2]$ is cyclic in $\mathcal D_\alpha$, then
 $
 \mathcal Z(p) \cap \mathbb T^2
 $
 is empty or finite.
 \item\label{point-2-through-anisotropic} If $\alpha > 2$ and $p \in \mathbb C[z_1, z_2]$ is cyclic in $\mathcal D_\alpha$, then
 $
 \mathcal Z(p) \cap \mathbb T^2=\varnothing.
 $
\end{enumerate}

Let $(\alpha_1,\alpha_2)\in\mathbb R^2$ be fixed. Recall that the anisotropic Dirichlet-type space $\mathfrak D_{(\alpha_1,\alpha_2)}$ consists of all holomorphic functions on the unit bidisk 
$\mathbb D^2$,
$
f(z,w)=\sum_{k,l=0}^{\infty} a_{kl} z^k w^l
$
such that
\[
\|f\|_{(\alpha_1,\alpha_2)}^2
:=\sum_{k,l=0}^{\infty}(k+1)^{\alpha_1}(l+1)^{\alpha_2}|a_{kl}|^2<\infty
\]
(see \cite[Subsection~1.1]{KKRS2019}). \\

\noindent \textit{Proof of \eqref{point-1-through-anisotropic}.} Let $\alpha>1$. If $p$ is cyclic in $\mathcal D_\alpha$ , then there exists a sequence of polynomials $(q_n)$ such that
$q_n p \to 1$ in $\mathcal D_\alpha$.
Since 
$\|\cdot\|_{(\alpha,0)}\leq \|\cdot\|_\alpha$ for $\alpha>0$, it follows that
$q_n p \to 1$ in $\mathfrak D_{(\alpha,0)}$
and $p$ is cyclic in  
$\mathfrak D_{(\alpha,0)}$.
Thus, \eqref{point-1-through-anisotropic} follows by applying \cite[Theorem~1]{KKRS2019}, which characterizes cyclicity in the anisotropic spaces $\mathfrak D_{(\alpha_1,\alpha_2)}$, with $(\alpha_1,\alpha_2):=(\alpha,0)$. In this case,
$
\alpha_1+\alpha_2=\alpha>1
$
and
$\min\{\alpha_1,\alpha_2\}=0\leq 1.
$\\

\noindent \textit{Proof of \eqref{point-2-through-anisotropic}.}
Let $\alpha>2$. Observe that for $\alpha>0$,
\[
(k+1)^{\alpha/2}(l+1)^{\alpha/2}\leq  (k+l+1)^\alpha,
\]
and hence
$
\mathcal D_\alpha \subseteq \mathfrak D_{(\alpha/2,\alpha/2)}.
$
Now let $p\in \mathbb C[z_1,z_2]$. As in the proof of \eqref{point-1-through-anisotropic}, if $p$ is cyclic in $\mathcal D_\alpha$, then it is also cyclic in
$
\mathfrak D_{(\alpha/2,\alpha/2)}.
$
Applying \cite[Theorem~1]{KKRS2019} with
$
(\alpha_1,\alpha_2)=
\left(\alpha/2,\alpha/2\right),
$
yields \eqref{point-2-through-anisotropic}.

\begin{remark}
By \cite[Proposition~4.5]{BS2026}, for $0<\alpha<2$, we have
\[
\mathcal D_{\alpha}
=
\mathfrak D_{(\alpha,0)}\cap \mathfrak D_{(0,\alpha)}.
\]
\end{remark}

\section{General properties of cyclic functions}\label{gpcf}
 In this section, we discuss some general properties of cyclic functions.
\subsection{Slices of a function.}  
\label{criterion-1}
In this subsection, we establish a result on the cyclicity of
univariate slices of a function, which serves as a quick tool for determining whether a function is a suitable candidate for being cyclic.

Let $f = f(z_1,z_2)$ be a holomorphic function on the bidisk. 
By fixing one variable, say $z_1$, the slice
\[
f_{z_1}:\mathbb D\to \mathbb C,\quad
f_{z_1}(z_2) \;:=\; f(z_1,z_2),
\]
 is a holomorphic function on the unit disk. 
The slice $f_{z_2}$ is defined analogously.

\begin{proposition}
If $f$ is cyclic in $\mathcal{D}_\alpha$, then $f_{z_1}$ and $f_{z_2}$ are cyclic in ${D}_{\alpha}.$ 
\end{proposition}
\begin{proof}
Let $\alpha \geq 0$ and $z_1 \in \mathbb{D}$. Consider
\[
\begin{aligned}
\|f_{z_1}\|_{D_\alpha}^2
&\underbrace{=}_{\eqref{norm-univariate}} 
\sum_{j \ge 0} (j+1)^\alpha \left| \sum_{i \ge 0} a_{ij} z_1^i \right|^2 \\[4pt]
&= \sum_{j \ge 0} \left| \sum_{i \ge 0} (j+1)^{\alpha/2} a_{ij} z_1^i \right|^2 \\[4pt]
&\le \sum_{j \ge 0}
\left( \sum_{i \ge 0} (j+1)^{\alpha} |a_{ij}|^2 \right)
\left( \sum_{i \ge 0} |z_1|^{2i} \right) \\[4pt]
&\le \sum_{j \ge 0}
\left( \sum_{i \ge 0} (i+j+1)^{\alpha} |a_{ij}|^2 \right)
\left( \sum_{i \ge 0} |z_1|^{2i} \right) \\[4pt]
&\le \frac{1}{1 - |z_1|^2} \, \|f\|_{\alpha}^2,
\end{aligned}
\]
{where the first inequality follows from the Cauchy-Schwarz inequality.}

Now suppose that $\alpha < 0$. For a fixed $z_1 \in \mathbb{D}$, we have $f_{z_1} \in D_{\alpha}$. Observe that
\[
f_{z_1}(z_2) = f(z_1, z_2) = \sum_{i=0}^\infty f_i(z_1)\, z_2^i,
\]
and hence
\[
\|f_{z_1}\|_{D_\alpha}^2 = \sum_{i=0}^\infty (i+1)^\alpha \, |f_i(z_1)|^2.
\]
Since each $f_i \in D_\alpha$ and $D_\alpha$ is a reproducing kernel Hilbert space, we have
\[
f_i(z_1) = \langle f_i, k_{z_1} \rangle,
\]
where $k_{z_1}$ denotes the reproducing kernel at $z_1$. Therefore,
\[
\|f_{z_1}\|_{D_\alpha}^2 = \sum_{i=0}^\infty (i+1)^\alpha \, |\langle f_i, k_{z_1} \rangle|^2.
\]
Applying the Cauchy--Schwarz inequality, we obtain
\beqn
\|f_{z_1}\|_{D_\alpha}
\le \|k_{z_1}\|_{D_\alpha} \|f\|_{\mathfrak{D}_{\alpha}}
\underbrace{\le}_{\eqref{inclusions-2}}
\|k_{z_1}\|_{D_\alpha} \|f\|_{\mathcal{D}_{\alpha}}
\eeqn
(cf.\ \cite[Proposition~2.1]{BCLSS2015}).

Combining the norm estimates for both $\alpha \ge 0$ and $\alpha < 0$, we conclude that if $f$ is cyclic in $\mathcal{D}_{\alpha}$, then the slices $f_{z_1}$ and $f_{z_2}$ are cyclic in $D_{\alpha}$.
\end{proof}

A natural question is whether the converse of the above result holds. To examine this, consider \( p(z_1, z_2) = 2 - z_1 - z_2 \). Note that the slices of \( p \) are cyclic in \( D_{\alpha} \) for all \( \alpha \), but \( p \) itself is not cyclic in \( \mathcal{D}_{\alpha} \) for \( \alpha > 2 \) (by Theorem~\ref{main}).

\subsection{Diagonal restriction of a function.}
\label{criterion-2}
In this subsection, we establish a result on the cyclicity of the
diagonal restriction of a function, which is another
criterion for determining whether a function is a suitable candidate for being cyclic.

Given a holomorphic function $f$ on $\mathbb D^2,$ define  the diagonal restriction $\mathcal Of$ by $
 (\mathcal{O} f)(z) := f(z,z).$ Then $\mathcal{O} f$ is holomorphic on $\mathbb{D}.$
The following proposition gives a necessary condition for the cyclicity of a function in $\mathcal{D}_{\alpha}.$
The proof follows closely the proof of 
\cite[Proposition~2.2]{BCLSS2015}.

\begin{proposition}
    If $f$ is cyclic in $\mathcal{D}_{\alpha}$, then $\mathcal{O}f$ is cyclic in $D_{\alpha-1}.$
\end{proposition}

\begin{proof}
It suffices to prove that for
$f \in \mathcal{D}_\alpha, \alpha \in \mathbb R$, 
\begin{equation}
\label{diagonal-cyclicity}
    \|\mathcal{O} f\|_{D_{\alpha-1}} 
    \;\leq\; \|f\|_\alpha.
\end{equation}
Indeed, if $q_n f \to 1$ in $\mathcal{D}_{\alpha}$, then by \eqref{diagonal-cyclicity} we have
\[
\|\mathcal{O}(q_n)\mathcal{O}(f) - 1\|_{D_{\alpha-1}}=
\|\mathcal{O}(q_n f - 1)\|_{D_{\alpha-1}} \leq \|q_n f - 1\|_{\alpha}.
\]
Thus, $\mathcal{O}(f)$ is cyclic in $D_{\alpha-1}$. It remains to prove \eqref{diagonal-cyclicity}.

Let 
$
f(z_1,z_2) = \sum_{k=0}^\infty \sum_{l=0}^\infty a_{k,l} z_1^k z_2^l.
$
 Then
\[
\mathcal{O} f(z) = \sum_{k=0}^\infty \sum_{l=0}^\infty a_{k,l} z^{k+l}
\]
converges absolutely for every $|z| < 1$, hence $\mathcal{O}f$ can be rewritten as
\[
\mathcal{O} f(z) = \sum_{n=0}^\infty b_n z^n, 
\qquad b_n = \sum_{k+l=n} a_{k,l} = \sum_{k=0}^n a_{k,n-k}.
\]
Thus,
\begin{equation}
\label{diagonal-rest-estimate}
\|\mathcal{O} f\|^2_{D_{\alpha-1}} 
= \sum_{n=0}^\infty |b_n|^2 (n+1)^{\alpha-1}
= \sum_{n=0}^\infty \left| \sum_{k=0}^n a_{k,n-k} \right|^2 (n+1)^{\alpha-1}.
\end{equation}
By the Cauchy-Schwarz inequality, it follows that
\beqn \left| \sum_{k=0}^n a_{k,n-k} \right|^2  &\leq& \left(\sum_{k=0}^n \left| a_{k,n-k}\right|^2(n+1 )^{\alpha}\right)\sum_{k=0}^n (n+1)^{-\alpha} \\
 &=& \sum_{k=0}^n\left| a_{k,n-k}\right|^2(n+1 )^{\alpha}(n+1)^{-\alpha+1}.
\eeqn
Using this in \eqref{diagonal-rest-estimate}, it follows that
$$ 
\|\mathcal{O} f\|^2_{D_{\alpha-1}} 
 \leq \sum_{n=0}^\infty \sum_{k=0}^n\left| a_{k,n-k}\right|^2(n+1 )^{\alpha}
 =
 \|f\|^2_\alpha,
$$
proving
\eqref{diagonal-cyclicity}.
\end{proof}

\begin{remark}
We refer the reader to (\cite[Proposition~2.2]{BCLSS2015}) for analogous statement for $\mathfrak D_{\alpha}.$ 
\end{remark}

\section{Concluding remarks}\label{con-rem}
We provided a self-contained proof of the cyclicity of the polynomial $2 - z_1 - z_2$; we now give simple proofs of the cyclicity of $1 - z_1 z_2$ and $z_1-a, |a|\geq 1$ using standard techniques. These are model polynomials and often play a crucial role in establishing cyclicity results for broader classes of functions (e.g., see the proof of  \cite[Theorem 5]{KKRS2019}), and therefore merit particular attention to ensure independent proofs.

The following has been proved in \cite[Proposition~38]{Z25} using a dilation argument. We provide an alternative proof.
\begin{proposition}\label{one-var}
    Let $p(z)=z-a$ with $|a|\geq 1.$ Then $P(z_1,z_2):=p(z_1)$ is cyclic in $\mathcal D_\alpha$ if and only if $ \alpha \leq 1.$
\end{proposition}
\begin{proof}
  For $f \in D_\alpha$, define $F(z_1,z_2)=f(z_1).$ Note that
    $\|f\|_{D_\alpha}=\|F\|_{\mathcal D_\alpha}.$
  {Since it is well-known \cite{BS1984} that $p$ is cyclic in $D_{\alpha}$ if and only if $\alpha \leq 1,$} 
  the implication $(\Leftarrow)$ is clear. For the implication $(\Rightarrow)$, assume that $\alpha >1.$  Let $q_n \in \mathbb{C}[z_1,z_2]$ be a sequence such that $\|q_nF-1\|_{\alpha}\rightarrow 0.$ By the orthogonality of monomials in $\mathcal{D}_{\alpha},$ one can choose  $q_n(z_1,z_2)=:g_{n}(z_1)$ where $g_n\in \mathbb C[z].$ But since $\|q_nF-1\|_{\alpha}=\|g_np-1\|_{D_{\alpha}}$, this contradicts the one variable Brown-Shields theorem  for $\alpha > 1.$
\end{proof}
\begin{proposition}
    $1-z_1z_2$ is cyclic in $\mathcal D_\alpha$ if and only if $\alpha\leq1.$
\end{proposition}

\begin{proof}
Note that if \( f \) is a function of one variable and \( F \) is defined by 
\( F(z_1,z_2):=f(z_1 z_2) \), then the following inequalities hold:
\begin{align}
\label{f-F-f}
\|f\|_{D_{\alpha}} &\leq \|F\|_{\alpha} \leq 2^{\alpha}\|f\|_{D_{\alpha}}, 
\quad \text{for } \alpha \geq 0,\\
2^{\alpha}\|f\|_{D_{\alpha}} &\leq \|F\|_{\alpha} \leq \|f\|_{D_{\alpha}}, 
\quad \text{for } \alpha < 0.
\label{f-F-f-2}
\end{align}

Let \( \alpha \in \mathbb{R} \). Define 
\[
P(z_1,z_2) = 1 - z_1 z_2 \quad \text{and} \quad p(z) = 1 - z.
\]

Assume first that \( \alpha \geq 0 \) and that \( P \) is cyclic in \( \mathcal{D}_{\alpha} \). 
Let \( \{P_n\}_n \subset \mathbb{C}[z_1,z_2] \) be a sequence of polynomials such that
$
\|P_n P - 1\|_{\alpha} \to 0.
$
Using the orthogonality of monomials in \( \mathcal{D}_{\alpha} \), we may assume without loss of generality that
$
P_n(z_1,z_2) = p_n(z_1 z_2),
$
for some sequence of one-variable polynomials \(\{p_n\}_n \subset \mathbb{C}[z]\). 
By the first inequality in \eqref{f-F-f}, it follows that
$
\|p_n p - 1\|_{D_\alpha} \to 0,
$
and hence \(1 - z\) is cyclic in \(D_{\alpha}\). By the Brown--Shields theorem, this implies that \( \alpha \leq 1 \).

The converse direction follows by a similar argument, using the second inequality in \eqref{f-F-f}.

For \( \alpha < 0 \), the reasoning is analogous, with \eqref{f-F-f-2} used in place of \eqref{f-F-f}.
\end{proof}

We conclude the paper with a few remarks on capacity. Finite logarithmic capacity and Riesz $\alpha$-capacity play important roles in identifying non-cyclic functions in $\mathfrak D_\alpha$. This approach originates in the work of Brown and Shields \cite{BS1984} and was later extended to several variables by various authors. We refer the reader to \cite[Definition 2.2]{BKKLSS2016} for the definition of Riesz $\alpha$-capacity and to \cite[Definition 4.1]{BCLSS2015} for the definition of finite logarithmic capacity.

Note that if $f \in \mathcal D_{2}$, then $f \in \mathfrak D_{1}$. For functions $f \in \mathfrak D_{1}$, one can apply the notion of finite logarithmic capacity to determine whether $f$ is non-cyclic in $\mathfrak D_{1}$  \cite[Proposition 4.2]{BCLSS2015}, and consequently not cyclic in $\mathcal D_{2}$. A similar conclusion holds for $f \in \mathcal D_{\alpha}$ with $0<\alpha < 2$.

This discussion is summarized in the following result.

\begin{proposition}
Let $\alpha \in (0,2]$ and $f\in \mathcal D_{\alpha}.$ Let $f^*$ denote the radial limit of $f$, i.e., $f^{*}(e^{i\theta_1},e^{i\theta_2})= \lim_{r \to 1^-}f(re^{i\theta_1},re^{i\theta_2})$
wherever the radial limit exists.
Then the following holds:
\begin{itemize}
    \item[(a)] If $\alpha=2$ and $\mathcal Z(f^*)$  has positive logarithmic capacity, then $f$ is not cyclic in $\mathcal{D}_{2}$.
    \item[(b)] If $0<\alpha<2$ and $\mathcal Z(f^*)$  has positive Riesz $\alpha$-capacity, then $f$ is not cyclic in $\mathcal{D}_{\alpha}.$
\end{itemize}
\end{proposition}

\bigskip

\noindent{\bf Acknowledgment.}
{The authors express their sincere gratitude to the referees for their careful reading of the manuscript, their insightful comments, and their suggestions, which helped correct several errors in the original version.}
We also thank one of the referees for allowing us to include an alternative approach to proving non-cyclicity (see Subsection \ref{subsec:alternative-approach}).
\bigskip

\noindent{\bf Statements and Declarations}\\

\noindent{\bf Data Availability Statement:} Data sharing is not applicable to this article as
no datasets were generated or analyzed during the current study.

\noindent{\bf Ethics Declaration} The submitted work is original and was not published
elsewhere.

\noindent{\bf Competing Interests} All authors certify that they have no affiliations with
or involvement in any organization or entity with any financial interest or
non-financial interest in the subject matter or materials discussed in this
manuscript.

\end{document}